\documentclass[11pt]{article}

\usepackage[left=3cm, top=3cm, right=3cm, bottom=3cm ]{geometry}
\usepackage{authblk}
\usepackage[onehalfspacing]{setspace}

\usepackage[T1]{fontenc}

\usepackage[bookmarks=closed]{hyperref}
\usepackage{breakurl}

\usepackage{amsmath,amsfonts}
\usepackage{algorithmic}
\usepackage{algorithm}
\usepackage{array}
\usepackage[caption=false]{subfig}
\usepackage{textcomp}
\usepackage{stfloats} 
\usepackage{hyperref}
\usepackage{url}
\usepackage{verbatim}
\usepackage{siunitx}
\usepackage{graphicx}
\usepackage{amssymb, amsthm}
\usepackage{cite}

\usepackage{enumitem}
\usepackage{multirow}

\usepackage{todonotes}

\newtheorem{theorem}{Theorem}[section]
\newtheorem{lemma}[theorem]{Lemma}

\newtheorem{proposition}[theorem]{Proposition}				
\newtheorem{definition}[theorem]{Definition}

\usepackage{siunitx}

\usepackage{mathtools}
\usepackage{booktabs}
\usepackage{dsfont}
\usepackage{pgfplots}
\usetikzlibrary{intersections, pgfplots.fillbetween}
\definecolor{mycolor}{RGB}{85,125,250}
\definecolor{gcolor}{RGB}{0.01,0.199,0.1}

\tikzset{declare function = {
                aux(\x)= (\x <= 0) * (0) +
                          and(\x > 0, \x < 1) * ((3*(\x-1)^2-2*(\x-1)^3)/( 3*(\x-1)^2-2*(\x-1)^3 + (3*\x^2-2*\x^3))) +
                          (\x >= 1) * (1)
               ;
               v(\x)= (abs(\x) <= 1/3) * (1) +
                          and(abs(\x) >= 1/3, abs(\x) <= 2/3) * (cos(pi/2)*aux(3*abs(\x)-1))    +
                          (abs(\x) >= 2/3) * (0)
               ;
               left(\x)= (\x >= -1/3) * (0) +
                          and(\x <= -1/3, \x >= -2/3) * (cos(pi/2)*v(3*abs(\x)-1))    +
                          (\x <= -2/3) * (0)
               ;
               right(\x)= (\x <= 1/3) * (0) +
                          and(\x >= 1/3, \x <= 2/3) * (cos(pi/2)*v(3*abs(\x)-1))    +
                          (\x >= 2/3) * (0)
               ;
               mid(\x)= (abs(\x) <= 1/3) * (1) + (abs(\x) >= 1/3) * (0)
               ;
               left_adap(\x)= (\x >= -1/3) * (0) +
                          and(\x <= -1/3, \x >= -2/3) * (cos(pi/2)*v(3*abs(\x)-1))    +
                          (\x <= -2/3) * (0)
               ;
               right_adap(\x)= (\x <= 1/3) * (0) +
                          and(\x >= 1/3, \x <= 2/3) * (cos(pi/2)*v(3*abs(\x)-1))    +
                          (\x >= 2/3) * (0)
               ;
               mid_adap(\x)= (abs(\x) <= pi/18) * (1) + (abs(\x) >= pi/18) * (0)
               ;
               }}

\newcommand{\abs}[1]{\lvert#1\rvert}
\newcommand{\norm}[1]{\lVert#1\rVert}
\newcommand{\Ao}{\mathbf{A}}
\newcommand{\Fo}{\mathbf{F}}
\newcommand{\Ko}{\mathbf{K}}

\newcommand{\Ro}{\mathcal{R}}

\newcommand{\Sp}{\mathds S}
\newcommand{\regfilt}{\Phi}
\newcommand{\Y}{\mathds Y}
\newcommand{\Vo}{\mathcal{V}}
\newcommand{\qsv}{\kappa}
\newcommand{\Bo}{\mathcal{B}}
\newcommand{\Io}{\mathcal{I}}

\newcommand{\sparen}[1]{\{#1\}}		      
\newcommand{\paren}[1]{(#1)}

\DeclareMathOperator{\ran}{ran}

\newcommand{\inner}[1]{\left\langle#1\right\rangle}

\newcommand{\reg}{\mathcal{R}}

\newcommand{\signal}{f}
\newcommand{\data}{g}

\newcommand{\C}{\mathbb{C}}
\newcommand{\R}{\mathbb{R}}
\newcommand{\N}{\mathbb{N}}
\newcommand{\Z}{\mathbb{Z}}

\newcommand{\dom}{\mathcal{D}}

\newcommand{\al}{\alpha}

\newcommand{\la}{\lambda}
\newcommand{\La}{\Lambda}

\AtBeginDocument{

}
\newcommand*\diff{\mathop{}\!\mathrm{d}}

\DeclareMathOperator{\supp}{supp}

\usepackage{soul}
\colorlet{lred}{red!40}
\colorlet{lgreen}{green!40}
\colorlet{lblue}{blue!40}
\definecolor{bananamania}{rgb}{0.98, 0.91, 0.71}

\numberwithin{equation}{section}
\numberwithin{theorem}{section}
\numberwithin{figure}{section}

\allowdisplaybreaks

\title{Translation invariant diagonal frame decomposition for the Radon transform}

\date{August 7, 2023}

\author{Simon Göppel}

\affil{Department of Mathematics, University of Innsbruck\authorcr
Technikerstrasse 13, 6020 Innsbruck, Austria\authorcr
E-mail:  \texttt{simon.goeppel@uibk.ac.at}
 }
 
 \author{Markus Haltmeier}

\affil{Department of Mathematics, University of Innsbruck\authorcr
Technikerstrasse 13, 6020 Innsbruck, Austria\authorcr
E-mail:  \texttt{markus.haltmeier@uibk.ac.at}
 }

\author{Jürgen Frikel}

\affil{Department of Computer Science and Mathematics, OTH Regensburg\authorcr Galgenbergstra{\ss}e 32, 93053 Regensburg, Germany\authorcr
E-mail:  \texttt{juergen.frikel@oth-regensburg.de}
 }

\begin{document}

\maketitle

\begin{abstract} 
In this article, we address the challenge of solving the ill-posed reconstruction problem in computed tomography using a translation invariant diagonal frame decomposition (TI-DFD). First, we review the concept of a TI-DFD for general linear operators and the  corresponding filter-based regularization. We then introduce the TI-DFD for the Radon transform on $L^2(\R^2)$ and provide an exemplary construction using the TI wavelet transform.
Presented numerical results clearly demonstrate the benefits of our approach over non-translation invariant counterparts.

\medskip\noindent \textbf{Keywords:}  
DFD, translation invariance, Radon transform,  inverse problems, regularization

\end{abstract}

\section{Introduction}
Computed tomography (CT) is a widely applied imaging modality in medicine and industry, where the underlying mathematical model is the Radon transform. For a function or signal $\signal \colon \R^2 \to \R$, the Radon transform is denoted by $\Ro \signal \colon \Sp^1 \times \R \to \R$.
It is well-known, that reconstructing a function $\signal$ from approximate knowledge of its line integrals amounts to an ill-posed inverse problem of the form $\data^\delta = \Ro \signal + \eta$ \cite{natterer}. Here, $\delta >0$ is some noise level and $\eta$ denotes the data distortions with $\norm{\eta}_2\leq \delta$. In particular, this means that inverting the Radon transform is unstable and that small perturbations in data can cause big reconstructions errors.

Classical filter based regularization is a well-known stabilization concept. Assuming a singular value decomposition (SVD) $\Ro \signal = \sum_{n\in\N} \sigma_n \inner{\signal, u_n} v_n$ for the Radon transform (in an appropriate function space setting \cite{louis1984,maass1991,QUINTO1983437,davison1981}), a regularized reconstruction is given by  $\signal_\alpha^\delta = \sum_{n\in\N} \regfilt_\alpha(\sigma_n) \inner{\data^\delta, v_n}u_n$, depending on a so-called regularizing filter $\regfilt_\alpha$, see  \cite{engl1996regularization,ebner2020regularization}. However, filtering based on the SVD comes with several shortcomings. In general, for an arbitrary linear operator, the SVD might be hard to compute numerically or not be known analytically. Additionally, the basis elements $u_n, v_n$ are only adapted to the operator itself, but not the underlying signal class of interest.

To overcome these limitations, the so-called diagonal frame decomposition (DFD) as a generalization of the SVD has been studied thoroughly in the recent years. In particular, DFDs are better suited as they not only can be adjusted to the underlying application and involved signals, but often also provide efficient implementations.  A prominent example of a DFD for the Radon transform is the wavelet-vaguelette decomposition (WVD) introduced in \cite{DONOHO1995101}.  Related construction involving curvelets and shearlets can for example be found in \cite{candes2002, COLONNA2010232}. A more general analysis of regularization properties and convergence results have been presented in \cite{ebner2020regularization, hubmer2022regularization, frikel2020frame}.

One drawback of the classical WVD reconstruction approach is that in general it lacks translation invariance, which can lead to well-known wavelet artifacts in the reconstruction \cite{mallat}. Translation invariant systems on the other hand are known to perform better in that regard for simple tasks such as denoising  \cite{nason1995stationary, coifman1995tidenoising} and have also been studied recently for a variational approach \cite{parhi2023cycle}. To overcome this,  in \cite{goeppel2022tidfd} the authors have introduced the concept of the translation invariant diagonal frame decomposition (TI-DFD) for general linear operators. Along with an analysis of the regularization properties of the filter based TI-DFD, the authors gave an exemplary construction of a TI-WVD regarding stable differentiation. These findings indicate improved regularization properties, when translation invariance is restored in the underlying wavelet system. Thus, the goal of the present paper is to construct a TI-WVD for the Radon transform. This way, we obtain an explicit filter based regularization strategy, which can be implemented efficiently. We will demonstrate its improved regularization properties by comparing the results to the classical WVD in a numerical example.\\

\noindent \textbf{Notation.} For $f\in L^2(\R^2)$, the Fourier transform is denoted by $\widehat f = \Fo f$, where $\widehat f (\xi) \coloneqq \int_{\R^2} f(x) e^{-i\inner{\xi, x}}$ if $f \in L^2(\R^d) \cap L^1(\R^d)$.  For $f \in L^2(\Sp \times \R)$, the one dimensional Fourier transform in  the second variable is be denoted by $\Fo_2$. Furthermore, we write $u^* (x) \coloneqq \overline{u(-x)}$, where $\overline{z}$ is the complex conjugate of $z \in \C$.

\section{The Translation Invariant Diagonal Frame Decomposition}
\label{sec:TI-DFD}

In this section, we recall the concept of translation invariant (TI)  frames  \cite{mallat} and  translation invariant diagonal frame decompositions (TI-DFDs) of linear operators \cite{goeppel2022tidfd}. Furthermore, we recall the concept of filtered regularization using  TI-DFDs.\\ 

\begin{definition}[TI-frame] \label{def:ti-frame}
Let $\Lambda$ be an at most countable index set. We call the family $ (u_\la)_{\la  \in \La} \in L^2(\R^d)^\La$  a translation invariant frame (TI-frame) for $L^2(\R^d)$ if for all $\la\in\La$ we have $\hat u_\la \in L^\infty(\R^d) $ and there exist constants $A,B>0$, such that
\begin{equation}\label{eq:TI-frame}
\forall f\in L^2(\R^d) \colon \quad A \norm{\signal}_2^2 \leq \sum_{\la  \in \La} \norm{ u^\ast_\la \ast \signal}_2^2 \leq B  \norm{\signal}_2^2  \,.
\end{equation}
We call a TI-frame  $ (u_\la)_{\la  \in \La}$ tight if \eqref{eq:TI-frame} holds with TI-frame founds $A=B=1$.
\end{definition}

Defining $w_\la \coloneqq \Fo^{-1} \paren{2\pi \widehat u_\la / \sum_{\mu \in \La} \abs{\widehat u_\mu}^2}$ for every TI frame element $u_\la$, we obtain the so-called canonical dual $(w_\la)_{\la\in\La}$ of $(u_\la)_{\la\in\La}$. It holds that
\begin{equation}\label{eq:reproduction formula}
    \forall \signal \in L^2 (\R^d) \colon \quad \signal = \sum_{\la\in\La} w_\la \ast \paren{u_\la^* \ast \signal}.
\end{equation}
Note that \eqref{eq:reproduction formula} in fact holds true for any dual frame $(w_\la)_{\la\in\La}$ defined by the property $\sum_{\la} \paren{\Fo w_\la} \cdot \paren{\overline{\Fo u_\la}} = 2\pi$. In particular, the canonical dual always exists but it not  uniquely defined by \eqref{eq:reproduction formula}. Further note  that TI-frame is not a frame in the classical sense. In particular, TI frame coefficients $u_\la^* \ast \signal (x) = \inner{\signal, u_\la (\cdot - x)}$ use a continuous  translation parameter $x$.\\

\begin{definition}[TI-DFD]\label{def:ti-dfd} Let $\Ko \colon \dom (\Ko) \subseteq L^2(\R^d) \to \Y$ be a closed linear operator,  where $\Y$ is a Hilbert space. We denote the of bounded operators between $\Y$ and $L^2(\R^d)$ by $B(\Y, L^2(\R^d))$. The system $ \paren{u_\la, \Vo^*_\la, \qsv_\la}_{\la  \in \La}$ is called a translation invariant frame decomposition (TI-DFD) for $\Ko$, if the following properties hold:
\begin{enumerate}[itemindent=1.5em, label=(TI\arabic*), topsep=0em]
\item \label{ti1} $ \paren{u_\la}_{\la  \in \La} \in L^2(\R^d)^\Lambda$ is a TI-frame for $L^2(\R^d)$.
\item \label{ti2} $\forall \la \in \La$ we have $ \Vo^*_\la \in  B(\Y, L^2(\R^d))$ and
\begin{equation*}
\forall g \in \overline{\ran{\Ko}} \colon \sum_{\la  \in \La} \norm{\Vo^*_\la \data}_2^2 \asymp \norm{\data}_{\Y}^2 \,.
\end{equation*}
\item \label{ti3} $\forall \la \in \La \colon  \qsv_\la \in (0, \infty)$ and
\begin{equation*}
\forall \signal\in \dom(\Ko) \colon  \Vo^*_\la (\Ko \signal)   = \qsv_\la \, (u^\ast_\la \ast \signal ).
\end{equation*}
\end{enumerate}
\end{definition}

Here, we define $F \asymp G \vcentcolon\Leftrightarrow \exists c_1,c_2 > 0 \colon c_1 G \leq F \leq c_2 G $.\\

The TI-DFD decomposes an operator into translation-invariant components, splitting the problem of recovering $\signal$ from $\Ko \signal = \data$ into several subproblems indexed by $\la$. Here, $(\Vo_\la^* \data)_{\la\in\La}$ are the coefficient functions given in the data domain, where \ref{ti2} ensures that this decomposition is stable in both directions. Property (TI3) states that the coefficient functions of the original signal can be recovered from the data with stability given by $\kappa_\lambda$. Together with \eqref{eq:reproduction formula} the definition of a TI-DFD immediately gives rise to the reproducing formula
\begin{equation}\label{eq:ti dfd recon}
\signal = \sum_{\la\in\La} w_\la \ast (\kappa_\la^{-1} \cdot \Vo_\la^* \data),
\end{equation}
for all $\signal \in \dom (\Ko)$ and $\data = \mathcal{K} \signal$. Note that the instability of inverting $\Ko$ is reflected via the quasi-singular values $(\kappa_\la)_{\la\in\La}$.  More precisely, in \cite{goeppel2022tidfd} it has been shown that the inverse operator $\Ko^{-1}$ is unbounded if and only if the quasi-singular values $(\kappa_\la)_{\la\in\La}$ accumulate at zero, assuming $\inf_\la \norm{\Vo_\la^*}_{\mathrm{op}} > 0$. \\

\begin{definition}[Regularizing filter] \label{def:regfilt} A family $(\regfilt_\al)_{\alpha>0}$ of piecewise continuous functions $\regfilt_\al \colon (0, \infty) \to \R$ is called a regularizing filter if the following hold:
\begin{enumerate}[itemindent=2em, leftmargin=1em,  label=(F\arabic*), topsep=0em]
\item\label{def:regfilt1}  $\forall \alpha >  0 \colon \norm{\regfilt_\al}_\infty < \infty$.
\item\label{def:regfilt2} $\exists C > 0 \colon \sup \{\abs{\qsv \regfilt_\al(\qsv) } \colon \alpha > 0 \wedge \qsv  \geq 0\} \leq C$.
\item\label{def:regfilt3} $\forall \qsv \in (0, \infty) \colon \lim_{\alpha \to 0} \regfilt_\al (\qsv) = 1/\qsv$.\\
\end{enumerate}
\end{definition}

The following theorem summarizes main results of \cite[Sections 2 and 3]{goeppel2022tidfd}.
For that recall  the notion of a regularization method \cite[Definition 3.1]{engl1996regularization}.\\

\begin{theorem}
Let $(u_\la,\Vo^*_\la, \qsv_\la)_{\la \in \La}$ be a TI-DFD for $\Ko$, let $(w_\la)_{\la \in \La}$ be a dual TI-frame for  $(u_\la)_{\la \in \La}$ and \begin{equation} \label{eq:ti-filt}
\reg^\regfilt_\al  g \coloneqq  \sum_{\lambda\in\Lambda}  w_\lambda \ast ( \regfilt_\al (\qsv_\lambda)  \cdot  (\Vo^*_\la g ) )
\end{equation}
where $(\regfilt_\al)_{\alpha>0}$ is a regularizing filter. Then we have:
\begin{enumerate}
\item  \label{eq:ti-recon}
$\forall \data \in \ran(\Ko) \colon  \quad \Ko^{-1}\data = \sum_{\lambda\in\Lambda}   w_\lambda  \ast ( \qsv_\lambda^{-1} \cdot (\Vo^*_\la \data) )$.
\item The family $(\reg^\regfilt_\al)_{\alpha>0}$  together with suitable parameter choice, defines a regularization method for inverting $\Ko$.\\
\end{enumerate}
\end{theorem}

\section{TI-DFD for the Radon Transform}
\label{sec:inversion of radon via ti-dfd}

Recall that for $f \in L^1(\R^2) \cap L^2(\R^2) $ the Radon transform $\Ro f \colon \Sp^1 \times \R \to \R$ is defined by
\begin{equation}\label{eq:R0}
\Ro \signal\paren{\theta, s} = \int_\R \signal \paren{s\theta + t \theta^\perp} \diff t \,,
\end{equation}
for almost  every  $ \paren{\theta, s} \in \Sp^1\times\R$.  In this section we extend the Radon transform to a closed operator between $L^2$ spaces  \cite{smith1977practical} and then construct corresponding TI-DFDs.

\subsection{The Radon transform on $L^2$}
\label{sec:R}
In this section, we introduce the Radon transform as an operator on between $L^2$-spaces. Proofs of the stated properties and a detailed discussion can be found in \cite{smith1977practical}. In what follows, we will make extensive use of the Fourier Slice theorem \cite{natterer} which states that for all  $f \in L^1(\R^2) \cap L^2(\R^2) $  and almost every  $(\theta, \sigma) \in \Sp^1 \times \R$ we have
\begin{equation}\label{eq:fst}
\Fo_2 \Ro \signal (\theta, \sigma) = \Fo \signal (\sigma \theta)  \,.
\end{equation}
In fact we use an extension of \eqref{eq:fst}  to the  natural domain of definition $\dom (\Ro) \supsetneq L^1(\R^2) \cap L^2(\R^2) $. To this end we define the operator $\Bo \colon \dom (\Bo) \subseteq L^2 (\R^2) \to L^2 (\Sp^1 \times \R)$ by
\begin{equation}\label{eq:B}
 	\forall  (\theta, \sigma) \in \Sp^1 \times \R
	\colon \quad
	\Bo f \paren{\theta, \sigma}\coloneqq  \widehat f \paren{\sigma \theta}
\end{equation}
According to \eqref{eq:fst}, $\Bo$ is the Fourier representation of the Radon transform on  $L^1(\R^2) \cap L^2(\R^2) $. \\

\begin{proposition}[Properties of $\Bo$]
\label{prop:B}
The operator $\Bo$ as defined above satisfies the following:
\begin{enumerate}[label=(\alph*)]
\item\label{prop3-B} $\dom \paren{\Bo} = \sparen{ f \colon  \norm{\cdot}_2^{-1/2}   f \in L^2 (\R^2)}$
\item\label{prop4-B} $\dom (\Bo^*) = \left\{ g \colon \vert \cdot \vert^{-1/2} g(\theta,\cdot) \in L^2(\Sp^1 \times \R) \right\}$
\item\label{prop5-B} $\dom(\Bo)$ and $\dom (\Bo^*)$ are dense in $L^2$, respectively.
\item\label{prop1-B} $\Bo$ is well-defined, linear, injective and unbounded.
\item\label{prop2-B} $
\Bo^* g(\xi) =\Vert \xi \Vert_2^{-1} g \left( \xi/ \Vert \xi \Vert_2, \Vert \xi \Vert_2 \right)$.\\
\end{enumerate}
\end{proposition}

\begin{definition}[Radon transform on $L^2$]\label{def:R}
The operator $\Ro  \colon \dom(\Ro) \subseteq L^2(\R^2) \to L^2 (\Sp^1 \times \R)$ defined as composition
\begin{equation}\label{eq:R}
\Ro f \coloneqq  (\Fo_2^{-1} \circ  \Bo \circ  \Fo)  (f)
\end{equation}
is called the Radon transform on $L^2$.\\
\end{definition}

\begin{proposition}[Properties of $\Ro$]\label{prop:properties K} The Radon transform $\Ro$ satisfies the following properties:
\begin{enumerate}[label=(\alph*)]
\item\label{prop2-K} $\Ro$ is well-defined, linear, injective and unbounded.
\item\label{prop1-K} $\Ro$ is the closed extension of the operator defined by \eqref{eq:R0}.
\item\label{prop3-K} $\dom (\Ro) = \sparen{ \signal \in L^2 (\R^2) \mid \Fo \signal \in \dom (\Bo)} $.
\item\label{prop4-K} $\dom (\Ro^*) = \sparen{\data \in L^2 (\Sp^1 \times \R) \mid \Fo_2 g \in \dom (\Bo^*)}$.
\item\label{prop5-K} $\dom(\Ro)$ and $\dom(\Ro^*)$ are dense in $L^2$, respectively.
\item  $\Ro^* = \Fo^{-1} \circ \Bo^* \circ \Fo_2$.\\
\end{enumerate}
\end{proposition}

Proposition \ref{prop:properties K} in particular states that the Radon transform $\Ro$ is an unbounded operator. On the other hand, the restriction of $\Ro$ to various closed subspaces is bounded. The most common case is the restriction of $\Ro$ to $L_D^2 (\R^2) \coloneqq \left\{ f \in L^2 (\R^2) \colon \supp (f) \subseteq D \right\}$. However, the restriction to functions with compact support  poses a-priori assumptions that  are not translation invariant. Further, with the above considerations we obtain the classical filtered backprojection formula (FBP) \cite{natterer} between $L^2$-spaces, i.e.
\begin{equation}\label{eq:fbp}
\Ro^{-1} \data = \frac{1}{4\pi}\paren{\Ro^* \circ \Io_1} \paren{\data},
\end{equation}
where $\data = \Ro \signal$ and $\signal \in \dom(\Ro)$ \cite{SOLMON197661, smith1977practical}.
Here, the operator $\Io_1\signal \coloneqq \Fo^{-1} \paren{\abs{\cdot}  \Fo  \signal}$ is known as  Riesz potential.

\subsection{Necessary conditions}
\label{sec:ti-dfd for radon}

Before constructing a TI-DFD for the Radon transform on $L^2(\R^2)$ we derive some necessary conditions.\\

\begin{lemma}[Necessary Conditions]\label{lem:V}
Let $u_\la \in L^2(\R^2)$, $v_\la \in L^2(\Sp^1 \times \R)$ and $\kappa_\la >0$   satisfy $\Ro^* v_\la^* = \kappa_\la u_\la^*$ and $\hat v_\la, \abs{\cdot}^{-1} \hat v_\la  \in L^\infty(\Sp^1 \times \R)$, and define $ \Vo_\la^* g \coloneqq \Ro^* (v_\la^* \ast g)$.
\begin{enumerate}
\item\label{lem:V1} $(\Vo_\la)_{\la\in\La}$ satisfies \ref{ti3}.
\item\label{lem:V2} If $(u_\la, \Vo_\la^*, \kappa_\la)_{\la\in\La}$ is a TI-DFD for $\Ro$, then
\begin{equation}\label{eq:necessary condition}
\forall \la\in\La \colon \quad \widehat v_\la (\theta, \sigma) = \kappa_\la \abs{\sigma} \widehat u_\la (\sigma \theta).
\end{equation}
\end{enumerate}
\end{lemma}

\begin{proof}\mbox{}
\ref{lem:V1})
Let $\signal \in \dom \paren{\Ro}$. Then  
\begin{align*}
\kappa_\la  \paren{u_\la^* \ast \signal} 
& = \Fo^{-1} \paren{\kappa_\la \widehat u_\la^* \cdot  \widehat \signal}  
\\& = 
\Fo^{-1} \paren{\Bo^* \widehat v_\la^* \cdot \norm{\cdot}_2 \Bo^* \Bo \widehat f}
\\ & = 
\Fo^{-1} \Bo^* \paren{\widehat v_\la^* \cdot \widehat{\Ro \signal}} 
\\& = 
\Fo^{-1}  \Bo^* \Fo_2 \paren{v_\la^* \ast_s \Ro \signal}
\\ & = 
\Ro^* \paren{v_\la^* \ast_s \Ro \signal} 
\\ &= 
\Vo_\la^\ast (\Ro \signal)\,. 
\end{align*}

\ref{lem:V2})
According to \ref{lem:V1}) and the definition of $\Vo_\la^*$ we have $\Ro^* (v_\la^* \ast \Ro \signal) = \kappa_\la (u_\la^* \ast \signal)$. By applying the Fourier transform on both sides and since $\dom(\Ro)$ is dense, we obtain \eqref{eq:necessary condition}.
\end{proof}

\subsection{Construction of the TI-DFD}
\label{sec:ti-dfd for radon}

Now, let $\paren{u_\la}_{\la\in\La}$ be a 2D (tensor product) TI wavelet frame with mother wavelet $ u\in L^2(\R^2)$. That is, we assume the multi-scale structure
\begin{equation}\label{eq:u}
	\forall (j, l) \in \La \coloneqq \Z \times \sparen{\mathrm{H}, \mathrm{V}, \mathrm{D}} \colon \; u_{j,l} (x) = 2^j u_{0, l} (2^jx)\,,
\end{equation}
where $j\in\Z$ is the scale index and $l \in  \sparen{\mathrm{H}, \mathrm{V}, \mathrm{D}}$ indicates the horizontal, vertical or diagonal mother wavelet, respectively.\\

\begin{theorem}[TI-WVD for $\Ro$]\label{theo:tifd for radon}
Let $(u_{j,l})_{j,l\in\La}$  be defined by \eqref{eq:u}, suppose $\supp ( \widehat u_{0,l} ) = \sparen{\xi \mid a \leq \norm{\xi}_2 \leq b}$ for some $a,b>0$, and for  $(j, l) \in \Z \times  \sparen{\mathrm{H}, \mathrm{V}, \mathrm{D}}$ define
\begin{align}
v_{j,l}   &\coloneqq 2^{-j/2}  \Io_1 \Ro u_{j,l}   \\
\Vo_{j,l}^* (g)  &\coloneqq  \Ro^* (v_{j,l}^* \ast g) \,.
\end{align}
Then the system $\paren{u_{j,l}, \Vo_{j,l}^*, 2^{-j/2} }_{j,l\in\La}$  defines a TI-DFD for $\Ro$, which we will call TI-WVD for the Radon transform.
\end{theorem}

\begin{proof}
For the proof it remains to verify that $(\Vo_{j,l}^*)_{j,l}$ satisfies \ref{ti2}. Let $\data \in \ran (\Ro)$, then by the Parseval identity and changing to polar coordinates inside the integral, we have
\begin{align*}
\norm{\Vo_{j,l}^* \data}_2^2
& = \norm{\Fo \paren{\Vo_{j,l}^* \data}}_2^2 \\
& = \int_{\Sp^1}\int_0^\infty \frac{1}{\sigma} \abs{ \widehat v_{j,l}^* (\theta, \sigma) }^2 \cdot \abs{\widehat \data (\theta, \sigma)}^2 \diff \sigma \diff \theta\\
& = \int_{\Sp^1}\int_0^\infty \frac{\sigma}{2^j} \abs{ \widehat u_{j,l} (\sigma \theta) }^2 \cdot \abs{\widehat \data (\theta, \sigma)}^2 \diff \sigma \diff \theta.
\end{align*}
Since $a 2^j \leq \sigma \leq b 2^j$ on the support of the integrand and since $(u_{j,l})_{j,l}$ is a TI frame, taking the sum on both sides yields $aA \leq \sum_{j,l} ||\Vo_{j,l}^* \data ||_2^2 \leq bB$, where $A,B>0$ are the TI frame constants of $(u_{j,l})_{j,l}$.
\end{proof} 

\section{Numerical Experiments}
\label{sec:numerical experiments}

In this section, we present numerical comparisons between the classical WVD \cite{DONOHO1995101, ebner2020regularization} and the TI-WVD for the Radon transform, introduced in this article. The essential step in implementing the filtered TI-DFD reconstruction formula \eqref{eq:ti-filt} is calculating $\Vo_{j,l}^* \data = \Ro^* \paren{v_{j,l}^* \ast \data}$. To this end, we use the explicit expression \eqref{eq:necessary condition} and calculate
\begin{equation*}
\Ro^*\paren{v_{j,l}^* \ast \data} = \kappa_{j,l} \paren{u_{j,l}^* \ast \Ro^* \Io_1\data}.
\end{equation*}
Note that $\Ro^* \Io_1$ is the filtered backprojection (FBP) as given in \eqref{eq:fbp}. Together with the the Tikhonov filter $\regfilt_\alpha (\kappa) \coloneqq \kappa/(\kappa^2 + \alpha)$, the regularized TI-DFD \eqref{eq:ti-filt} for the Radon transform is given as
\begin{align*}
\Ao_\alpha^\Phi \data & = \sum_{j,l} \frac{2^{-j}}{2^{-j}+\alpha} u_{j,l} \ast \paren{u_{j,l}^* \ast \Ro^* \Io_1g}.
\end{align*}
Since the wavelet transform and the FBP can be efficiently implemented, this results in an efficient algorithm of the regularized reconstruction. For the implementation, we used Python 3.9.12. The Radon  transform and in particular the FBP where implemented via the scikit-image package, version 0.19.2 \cite{scikit-image}. The (TI) wavelet transform was employed via the PyWavelets package, version 1.3.0 \cite{lee2019pywavelets}.\\

We use a discretized synthetic phantom $\signal \in \R^{256 \times  256}$ and  chose the maximum of $8$ decomposition levels using the Haar-wavelet as underlying frame. Note that the Haar-wavelet is not band-limited, and thus Theorem \ref{theo:tifd for radon} not applicable in this case. However, we expect the result of the theorem to hold under weaker assumptions, but we do not yet have proof of this. We added white Gaussian noise to the data $\data^\delta = \Ro \signal + \delta\eta$. Here, we chose $\delta = 0.05$ and $\eta \sim \norm{\Ro \signal}_\infty \mathcal{N}(0,1)$. To guarantee a fair comparison, we performed a parameter search to determine the optimal regularization parameter $\alpha>0$ for both methods. The parameter was optimized in terms of the relative $\ell^2$ reconstruction error $\norm{\signal_{\text{rec}}^\alpha - \signal}_2/\norm{\signal}_2$, where $\signal$ is the ground truth and $\signal_{\text{rec}}^\alpha$ is the reconstruction, depending on the selected parameter.\\

Numerical results are shown in Figure \ref{fig:wvd example}, which  clearly show that the TI-DFD approach outperforms the standard WVD. While the denoising property of both methods is evident, the decimated DFD suffers from the well-known block like artifacts which are due to the sub sampling step in the decimated wavelet decomposition.  Quantitatively, the relative $\ell^2$ reconstruction error for the WVD amounts to 0.054 and for the TI-WVD to 0.048.

\begin{figure}[htb!]
    \centering
    \subfloat[]{\includegraphics[width=0.3\textwidth]{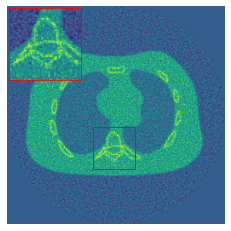}
        \label{fig:fbp}}
        \subfloat[]{\includegraphics[width=0.3\textwidth]{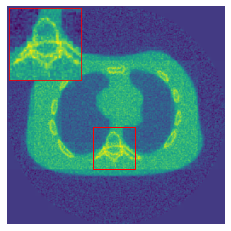}
        \label{fig:non ti frec haar}}
     	\subfloat[]{\includegraphics[width=0.3\textwidth]{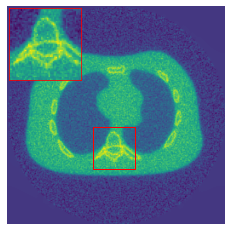}
        \label{fig:ti frec haar}}
\caption{Reconstructions from noisy Radon data $\data^\delta$ with $\delta=0.05$. \textbf{(a)} Filtered backprojection reconstruction using \eqref{eq:fbp}. \textbf{(b)},\textbf{(c)} Reconstructions via WVD and TI-WVD, respectively. Both reconstructions use the maximum number of $8$ levels of decomposition for the Haar-wavelet. In \textbf{(b)} the block like artifacts for the WVD are clearly visible in the magnified section.}
\label{fig:wvd example}
\end{figure}

\section{Conclusion}
In this article we presented  the concept of the translation invariant frame decomposition (TI-DFD) for the solution of linear operator equations. Subsequently, we constructed a TI wavelet-vagulette decomposition  (TI-WVS)  for the Radon transform as an instance of the TI-DFD. An advantage of classical frame decompositions is the translation invariance of the system which also has been demonstrated numerically. We have seen, that the use  of translation invariant frames leads to improved reconstructions when compared to classical frames.

\section*{Acknowledgments}

The contribution by S.\, G. is part of a project that has received funding from the European Union’s Horizon 2020 research and innovation program under the Marie Sk\l{}odowska-Curie grant agreement No 847476. The views and opinions expressed herein do not necessarily reflect those of the European Commission.

\end{document}